\documentclass[10pt,a4paper]{amsart}
\usepackage{graphicx}
\usepackage{amsaddr}

\usepackage{amsmath, cite}%
\usepackage{amssymb,indentfirst,calc,mathrsfs,euscript,bbm}%
\usepackage{dsfont}

\usepackage{setspace}%
\usepackage[reals]{layout}%
\usepackage{xr}%
\usepackage{amscd}%
\usepackage{comment}
\usepackage{epstopdf}
\usepackage{color}

\def\eps{{\varepsilon}}

\newtheorem{theorem}{Theorem}
\newtheorem{proposition}[theorem]{Proposition}
\newtheorem{lem}[theorem]{Lemma}
\newtheorem{corollary}[theorem]{Corollary}

\theoremstyle{definition}

\theoremstyle{remark}
\newtheorem{remark}[theorem]{Remark}

\begin{document}
\title[Semi-algebraic sets and equilibria]
{Semi-algebraic sets and equilibria of binary games}
\author{Guillaume Vigeral (Corresponding author)}
\address{Universit\'e Paris-Dauphine, CEREMADE, Place du Mar{\'e}chal De Lattre de Tassigny, 75775 Paris cedex 16, France.  vigeral@ceremade.dauphine.fr}

\author{Yannick Viossat}
\address{Universit\'e Paris-Dauphine, CEREMADE, Place du Mar{\'e}chal De Lattre de Tassigny, 75775 Paris cedex 16, France. viossat@ceremade.dauphine.fr}

%

%
%

\date{}

%

\maketitle



\begin{abstract}

Any nonempty, compact, semi-algebraic set in $[0,1]^n$ is the projection of the set of mixed equilibria of a finite game with 2 actions per player on its first $n$ coordinates.  A similar result follows for sets of equilibrium payoffs. The proofs are constructive and elementary.\\


\end{abstract}

\section{Introduction}

In bimatrix games, the structure of the set of Nash equilibria is relatively well understood: this is a finite union of convex polytopes (Jansen, 1981 \cite{Ja81}). Moreover, the possible sets of Nash equilibrium payoffs have been characterized by Lehrer et al. (2011) \cite{LSV11}:  a subset $E$ of $\mathbb{R}^2$ is the set of Nash equilibrium payoffs of a bimatrix game if and only if this is a finite union of rectangles with edges parallel to the axes; that is, of the form: $E= \cup_{1 \leq i \leq m} [a_i, b_i] \times [c_i, d_i]$, where $m \in \mathbb{N}$ and $a_i$, $b_i$, $c_i$, $d_i \in \mathbb{R}$, with $a_i \leq b_i$, $c_i \leq d_i$. 

For finite games with 3-players or more, the picture is much less clear. It is easily seen that the set of Nash equilibria or of Nash equilibrium payoffs is nonempty, compact and semi-algebraic; however, which semi-algebraic sets really arise as sets of Nash equilibria or of Nash equilibrium payoffs is not known. A few results have been obtained. For instance, 
Datta (2003) \cite{Da03} showed that any real algebraic variety is isomorphic to the set of completely mixed Nash equilibria of a 3-player game, and also to the set of completely mixed equilibria of an $N$-player game in which each player has two strategies. More recently, Balkenborg and Vermeulen (2014, Theorem 6.1) \cite{BV14} 
 showed that any nonempty connected compact semi-algebraic set is homeomorphic to a connected component of the set of Nash equilibria of a finite game in which each player has only two strategies, all players have the same payoffs, and pure strategy payoffs are either $0$ or $1$. These results show that, modulo isomorphisms or homeomorphisms, and a focus on completely mixed equilibria or connected components of equilibria, all algebraic or nonempty compact semi-algebraic sets may be encoded as sets of Nash equilibria. We provide another result in this direction. 

Since the set of Nash equilibria of an $N$-player finite game is a nonempty compact semi-algebraic subset 
of some $\mathbb{R}^k$, it follows from Tarski-Seidenberg's theorem that the projection of such a set on a subspace 
$\mathbb{R}^n$, $n<k$, satisfies the same properties. We prove a kind of converse of this fact: for any nonempty compact 
semi-algebraic set $E$, there exists a finite game with $N > n$ players, each having only two pure strategies, such 
that $E$ is precisely the projection of the set of Nash equilibria of this game on its first $n$ coordinates (those of the first $n$ players). 
In this statement, we see a mixed strategy of an $N$-player game with two strategies per player as a vector $(x_1,...,x_N)$ in $[0,1]^N$; that is, we identify the strategy of the $i^{th}$ player with the probability $x_i$ that it assigns to the first of its two strategies.

The above result implies a similar result on equilibrium payoffs, as opposed to equilibria:  for any nonempty compact semi-algebraic set $E$ in $\mathbb{R}^n$, there exists a finite game with 
$N > n$ players, each having only two pure strategies, such that $E$ is precisely the set of Nash equilibrium payoffs of the first $n$ players; that is, the projection of the set of Nash equilibrium payoffs on its first $n$ coordinates (
as will become clear, the ``first $n$ players" in our result on equilibrium payoffs have payoffs given by affine transformations of the strategies of the ``first $n$ players" in our result on equilibria). As discussed further in the next section, the result on equilibria has been obtained 
independently by Yehuda John Levy \cite{LE15}, who also obtained more general results on semi-algebraic functions and correspondences, but our techniques and precise results are different. 

Some differences with Datta (2003) \cite{Da03} and Balkenborg and Vermeulen (2014) \cite{BV14} should be stressed. First, in our result, there is no isomorphism or homeomorphism involved, but a projection on the first $n$ players.
Second, our results do not concern the set of completely mixed Nash equilibria, or a connected component of equilibria, but the whole set of equilibria.  These are not related to algebraic varieties or to connected semi-algebraic sets, but to (nonempty compact) semi-algebraic sets, which need not be connected. Also note that there is a fundamental difference between the set of completely mixed Nash equilibria and the set of Nash equilibria: the first may be empty while the second cannot. This represents a conceptual difficulty: any construction needs to check at some point the nonemptiness of the input set.
Third, our proofs are fully elementary. To be more precise, given a set and certificates of semi-algebraicity, closedness and nonemptiness, our construction does not use any results from real algebraic geometry. Starting with a game with $n$ players with two strategies each and choosing their first strategies with probabilities $x_1$, ..., $x_n$, we show how to add additional players with two strategies such that in equilibrium, these additional players choose their first strategies with probabilities that are powers of the $x_i$, and how yet additional players then allow to build and combine any polynomial in $x_1$,..., $x_n$ in order to obtain that the set of equilibrium strategies of the initial $n$ players is a given (nonempty compact) semi-algebraic set in $[0, 1]^n$. A small modification of the game then allows to obtain a given (nonempty compact) semi-algebraic set in $\mathbb{R}^n$ as the set of Nash equilibrium payoffs of $n$ players of this game. Note that, by contrast with the work of Yehuda John Levy, we provide a bound on the number of additional players in our construction, which is not far from being tight. Once again, this bound is obtained only by elementary arguments.
Finally, when the semi-algebraic set is defined by polynomials with integer coefficients, we prove a more precise result ensuring that the constructed game has integer pure payoffs, at the cost of additional players. Once again the construction (given certificates) is elementary, but the bound on the number of players and on the size of the integer payoffs depends on precise results of algebraic geometry. 

Note that if additional players are not restricted to have only two actions, Yehuda John Levy proved that only three such additional players are needed \cite{LE15}.

The remainder of this note is organized as follows: we introduce some definitions and prove our main results in Section \ref{sec:main}. Section \ref{sec:comp} compares our work to that of Yehuda John Levy. Extensions are discussed in Section \ref{sec:ext}. Section \ref{sec:rk} concludes. 
\section{Definition and main results}
\label{sec:main}
A subset of $\mathbb{R}^n$ is semi-algebraic if it may be obtained by finitely many unions and intersections of sets defined by a polynomial equality or strict inequality. By the finiteness theorem (see for example Proposition 5.1 in \cite{BE80}), a closed semi-algebraic subset of $\mathbb{R}^n$ may also be described by finitely many unions and intersections of sets defined by a polynomial \emph{weak} inequality. In particular, a compact subset $E$ of $\mathbb{R}^n$ is semi-algebraic if and only if there exist positive integers $A$ and $B$ and polynomials in $n$ variables $P_{ab}$, $1 \leq a \leq A$, $1 \leq b \leq B$ such that: 
\begin{equation}
\label{eq;semialg}
E=\bigcup_{a=1}^A \bigcap_{b=1}^B \{x \in \mathbb{R}^n,\ P_{ab}(x)\leq0\}
\end{equation} 
Let us say that a game is binary if each player has only two pure strategies. Note that our notion of binary games is weaker than the notion used by Balkenborg and Vermeulen (2014) \cite{BV14}: they define a game to be binary if each player has two pure strategies, and if in addition, this is a common interest game (all players have the same payoff), with pure strategy payoffs always equal to $0$ or $1$. 

Our first result is on equilibria. As before, we identify in its statement the mixed strategy of the $i^{th}$ player with the probability $x_i$ that it assigns to the first of her two pure actions. 
\begin{proposition}\label{prop:main2}
If $E$ is a nonempty compact semi-algebraic subset of $[0,1]^n$, then there exists an $N$-player binary game (with $N > n$) such that the projection of its set of Nash equilibria on its first $n$ coordinates (those of the first $n$ players) is equal to $E$. 
\end{proposition}
Our second result, a byproduct of our proof of Proposition \ref{prop:main2}, is on equilibrium payoffs. 
\begin{proposition}\label{prop:main1} If $E$ is a nonempty compact semi-algebraic subset of $\mathbb{R}^n$, then there exists an $N$-player binary game (with $N > n$) such that the projection of its set of Nash equilibrium payoffs on its first $n$ coordinates is equal to $E$. 
\end{proposition}
Essentially the same results have been independently obtained by Yehuda John Levy \cite{LE15}, but our techniques are different. Moreover, while the results of Yehuda John Levy are stronger in that Proposition \ref{prop:main2} appears as a corollary of a more general result on semi-algebraic functions and correspondences, our proof is more elementary and we obtain a bound on the number of players needed: if $E$ is described by (\ref{eq;semialg}), then both in Propositions \ref{prop:main2} and \ref{prop:main1}, \[N\leq 1+AB + n(3+2\ln_2(d))\] where $d$ is such that each $P_{ab}$ is of degree at most $d$ in each variable.  

The proof is constructive. It relies on appropriate gadget games, in the sense of algorithmic game theory. Before introducing these gadgets, we need to clarify our notation. We only consider binary games and we denote the two pure strategies of each player by Top and Bottom. It will be convenient to use the same piece of notation for the name of a player and its probability to play Top except that, to be able to distinguish between players and strategies, we use uppercase letters for players. Thus, the $n$ basic players of the game bear the admittedly unusual names of players $X_1$, $X_2$,..., $X_n$, and $x_i$ is the probability that player $X_i$ plays Top. The players we need to add are called players $X_{ik}$, $Y_{ik}$, $S_{ab}$ or $U$. Player $X_{ik}$ will be such that, in equilibrium, its probability $x_{ik}$ of playing Top is equal to $(x_i)^{2^k}$. Since any positive integer is a sum of powers of $2$, products of $x_{ik}$ allow to obtain any power $(x_i)^q$ as the value in equilibrium of a multiaffine function of the probabilities used by the players of the game (where by multiaffine, we mean affine in each variable). Adding and multiplying such quantities allow to obtain the quantities $P_{ab}(x)$, where $x=(x_1,...,x_n)$, as the value in equilibrium of the probability that an additional player $S_{ab}$ plays Top. An additional gadget game then forces an additional player $U$ to play Top when $(x_1,\cdots,x_n)\notin E$. Finally, the payoffs of the original players $X_1,..., X_n$ are defined in such a way that, at any equilibrium in which $U$ plays Top, $(x_1,...,x_n)=(z_1,...,z_n)$ where $z=(z_1,...,z_n)$ is a fixed arbitrary element of the nonempty set $E$. Hence, at any equilibrium, if $x\notin E$ then $x=z \in E$ a contradiction. The converse, that is the fact that each $x\in E$ appears in an equilibrium, is an easy byproduct of the construction.

We define our binary games by giving the payoff of each player when she chooses Top or Bottom, as a function of the \emph{mixed} strategy profiles of her opponents (more precisely, of their probabilities to play Top). These expressions will always be multiaffine in the probabilities of playing Top of the opponents (that is affine with respect to the probability to play Top of each opponent), ensuring that they correspond to payoffs in the mixed extension of a binary game. For instance, our first gadget game will have (at least)  two players, say $X_{\alpha}$ and $X_{\beta}$, playing Top with probability $x_{\alpha}$ and $x_{\beta}$ respectively,  and with payoffs if they play Top or Bottom described by the following tables: 
 \begin{equation}
\label{eq:lm1}
\mbox{Player } X_{\alpha} \quad 
\begin{array}{c|c} 
T  & x \\ 
\hline
B & x_{\beta}
\end{array} 
\qquad \quad
\mbox{Player } X_{\beta} \quad 
\begin{array}{c|c} 
T  &  x_{\alpha} \\ 
\hline
B & x
\end{array} 
\end{equation}
where $x$ is a real number in $[0,1]$. This means that the payoffs of players $X_{\alpha}$ and $X_{\beta}$ are respectively $x_{\alpha} x +(1- x_{\alpha}) x_{\beta}$ and $x_{\beta} x_{\alpha} + (1- x_{\beta})x$. Letting Player $X_{\alpha}$ be the row player and Player $X_{\beta}$ the column player, the corresponding payoff matrix would be 


\[   \begin{array}{cc}
     &   \begin{array}{cc}
    T \quad & B  \\
  \end{array} \\
      \begin{array}{c}
    T \\
    B \\
  \end{array} &   \left(\begin{array}{cc}
    x,1 & x,x  \\
    1,0 & 0,x  \\
  \end{array}\right) \\
  \end{array}
\]

\noindent independently of any other player's action.

\begin{lem}\label{basic}
Let $x \in [0,1]$. In any game with 2 players whose payoffs are given by \eqref{eq:lm1}, and an arbitrary number of other players with unspecified payoffs, we have: in any equilibrium, $x_{\alpha} = x$. 
\end{lem}
\begin{proof}
Assume $x_{\alpha}>x$. Then Player $X_{\beta}$ plays Top, hence $x_{\beta}=1$. So $x_{\beta}>x$,  thus Player $X_{\alpha}$ plays Bottom and $x_{\alpha}=0$. Since $x \in [0,1]$, this contradicts $x_{\alpha}>x$. Similarly if $x_{\alpha}<x$, then $x_{\beta}=0$ hence $x_{\alpha} =1$, a contradiction.
\end{proof}

\begin{lem}\label{basicbis}
Let $i \in \{1,..., n\}$. Consider a game with $n$ basic players $X_1$,...,$X_n$, and (at least) $2q$ additional players $X_{ik}$, $Y_{ik}$, $0 \leq k \leq q-1$, whose payoffs are given by the following tables (where, as in all subsequent payoff tables, $T$ and $B$ correspond to the actions of the player whose payoff is given in the table). 
\begin{equation}
\label{eq:xik}
\mbox{Player }  X_{ik} \quad 
\begin{array}{c|c} 
T  & \begin{array}{c} x_i \prod_{j=0}^{k-1} x_{ij} \vspace{0.1cm} \end{array} \\ 
\hline
B & y_{ik} 
\end{array} 
\qquad \quad
\mbox{Player } Y_{ik} \quad 
\begin{array}{c|c} 
T  & x_{ik} \\ 
\hline
B & \begin{array}{c}  x_i \prod_{j=0}^{k-1} x_{ij} \end{array}
\end{array} 
\end{equation}
In any equilibrium, for any $k$ in $\{0,...,q-1\}$, $x_{ik} = (x_i)^{2^k}$.
\end{lem}

\begin{proof} The proof is by induction on $q$. The case $q=1$ is Lemma \ref{basic} applied to players $X_{i0}$ and $Y_{i0}$ with 
$x_{\alpha}=x_{i0}$, $x_{\beta}= y_{i0}$, and $x=x_i$. Assume the result is true when adding $2(q-1)$ players. 
To show that the result is still true when adding $2q$ players, it suffices to show that we then have in equilibrium
$x_{i(q-1)}=(x_i)^{2^{q-1}}$. But letting $k=q-1$, the induction hypothesis and Lemma \ref{basic} applied to players $X_{ik}$ and $Y_{ik}$ with 
$x_{\alpha}= x_{ik}$, $x_{\beta}=y_{ik}$ and $x= x_i \prod_{j=0}^{k-1} x_{ij}$ show that in any equilibrium 
\[x_{ik}= x_i \prod_{j=0}^{k-1} x_{ij} = x_i \prod_{j=0}^{k-1} (x_{i})^{2^j}= (x_i)^{2^{k}}\] as required.
\end{proof}

The following basic lemma transforms a polynomial in $x_1$,..., $x_n$ into a multiaffine map in the $(x_i)^{2^k}$, $i=1$ to $n$, $k=0$ to $q-1$. Due to the previous lemma, this will allow us, for any polynomial $P_{ab}$, to add a player whose payoff at equilibrium when playing Top is equal to $P_{ab}(x_1,...,x_n)$.

\begin{lem}\label{lm:fab}
Let $P$ be a polynomial in the variables $x_1,\cdots,x_n$, with degree strictly less than $2^{q}$ in each variable. Then there exists a multiaffine map $f : \mathbb{R}^{nq} \to \mathbb{R}$ 
such that for every $(x_1,\cdots,x_n)\in \mathbb{R}^n$ $$P(x_1,\cdots,x_n)=f(x_1^{2^0},..., x_1^{2^{q-1}}, ..., x_n^{2^0},...,  x_n^{2^{q-1}}).$$ 
\end{lem} 
\begin{proof} Recall that any positive integer $m$ may be written as a sum of powers of $2$ (this is the binary development of $m$). 
Thus if $m < 2^{q}$, there exist numbers $\eps_0(m)$,..., $\eps_{q}(m)$ in $\{0, 1\}$ such that $m=\sum_{k=0}^{q} \eps_k (m) 2^k$ hence \[x_i^m= \prod_{k=0}^{q-1}(x_i^{2^k})^{\eps_k (m)}.\]
Thus if we replace any $x_i^m$ that appears in the polynomial $P$ by $\prod_{k=0}^{q-1} (x_i^{2^k})^{\eps_k}$, and denote the resulting expression by $f(x_1^{2^0},..., x_1^{2^{q-1}}, ..., x_n^{2^0},...,  x_n^{q-1})$, this defines a map $f: \mathbb{R}^{nq} \to \mathbb{R}$ such that $P(x_1,\cdots,x_n)=f(x_1^{2^0},..., x_1^{2^{q-1}}, ..., x_n^{2^0},...,  x_n^{q-1})$ for every $(x_1,\cdots,x_n)\in \mathbb{R}^n$. Since every $\eps_k (m)$ is either 0 or 1 the map $f$ is affine in each variable.
\end{proof} 


\noindent \textbf{Proof of Proposition \ref{prop:main2}. } 
Let $d < 2^q$. Start with a game with $n$ basic players $X_1$,..., $X_n$, 
and, for each $i \in \{1,...,n\}$, $2q$ additional players with payoffs \eqref{eq:xik}. Denote $\hat{x}=(x_{10},..., x_{1(q-1)}, ..., x_{n0},..., x_{n(q-1)})$. For any polynomial $P_{ab}$, define $f_{ab}$ as in Lemma \ref{lm:fab}. Remark that by Lemmas \ref{basicbis} and \ref{lm:fab}, at equilibrium $P_{ab}(x_1,\cdots,x_n)=f_{ab}(\hat{x})$. Since $f_{ab}$ is multiaffine, we may add a player $S_{ab}$, playing $T$ with probability $s_{ab}$, and with payoff 
\[
\mbox{Player $S_{ab}$ }\quad \begin{array}{c|c} 
T  &  f_{ab}(\hat{x}) \\ 
\hline
B & 0
\end{array} 
\]
Add also a player $U$, playing Top with probability $u$, and with payoff
\[
\mbox{Player $U$ }\quad \begin{array}{c|c} 
T  &  \begin{array}{c} \prod_{a=1}^A\sum_{b=1}^B s_{ab} \vspace{0.1cm}\\\end{array} \\ 
\hline
B & 0
\end{array} 
\]
Finally, let $z=(z_1,\cdots,z_n)$ be any element of $E$ and give to the initial $n$ players (players $X_i$) the payoffs 
\[
\mbox{Player $X_i$ }\quad \begin{array}{c|c} 
T  &  (z_i-x_{i 0})u \\ 
\hline
B & 0
\end{array} 
\]
Thus the game has (1+AB) additional players, plus $2(1+ \ln_2(d))$ for each player to construct the $X_i$, hence
\[N\leq 1+AB + n(3+ 2\ln_2(d)).\] 
More precisely, if $d(i)$ is the maximal degree in $x_i$ of the polynomials $P_{ab}$, then we need at most $N= 1+AB + 3n+ 2\sum_i \ln_2(d(i))$ players.
 
Let  $x=(x_1,\cdots,x_n)$ be the $n$ first coordinates of an equilibrium and assume $x\notin E$. At equilibrium, $f_{ab}(\hat{x})=P_{ab}(x)$. Moreover,  since $x\notin E$, for each $a$, there is a $b$ such that $P_{ab}(x)>0$ and thus $s_{ab}=1$. Hence 
\[\prod_{a=1}^A\sum_{b=1}^B s_{ab}>0\]  
and $u=1$ as well. Since $z\in E$, there must be an $i$ such that $x_i\neq z_i$. Recall that in any equilibrium $x_i=x_{i 0}$. However, $x_{i0 }< z_i $ implies $(z_i-x_{i 0})u>0$ and $x_i=1\geq z_i $, while  $x_{i0} > z_i $ implies $(z_i-x_{i 0})u<0$ and $x_i=0\leq z_i $. We get a contradiction in either case.

Let now $x$ be in $E$. Consider the following profile : $x_{ik}=y_{ik}=(x_i)^{2^k}$ for all $i$ and $k$, $u=0$, $s_{ab}=0$ if $P_{ab}(x)\leq 0$ and 1 otherwise. All players $X_{ik}$ and $Y_{ik}$ are indifferent, and players $X_i$ as well since $u=0$. Players $s_{ab}$ have no profitable deviations by construction. Finally, since $x\in E$, there exists $a$ such that for every $b$, $s_{ab}=0$, hence 
$$
\prod_{a=1}^A\sum_{b=1}^B s_{ab}=0$$ and Player $U$ is indifferent.\\

\begin{remark}
Any compact semi-algebraic set $E$ may be written as in (\ref{eq;semialg}), but it may be that $E$ is naturally described in some other way, and that putting it in form (\ref{eq;semialg}) is computationally expensive.  
Thus, it is interesting to note that if $E=\cap_{a=1}^A\cup_{b=1}^B \{x \in \mathbb{R}^n,\ P_{ab}(x)\leq0\}$, exactly the same proof would work, replacing the payoff of Player $U$ by
\[
\begin{array}{c|c} 
T  &  \begin{array}{c} \sum_{a=1}^A\prod_{b=1}^B s_{ab} \vspace{0.1cm}\\\end{array} \\ 
\hline
B & 0
\end{array} 
\]

More generally, assume $E$ is defined by unions and intersections (in any order) of $C$ sets $E_c=\{x,\ P_{c}(x)\leq0\}$. In this definition, replace any $E_c$ by $p_c$, any $\cap$ by a $+$, any $\cup$ by a $*$, and denote by $Q(p_1,\cdots,p_C)$ the resulting expression. Then the same proof, replacing the payoff of Player $U$ by
\[
\begin{array}{c|c} 
T  &  \begin{array}{c} Q(p_1,\cdots,p_C) \vspace{0.05cm}\\\end{array} \\ 
\hline
B & 0
\end{array} 
\]
shows that  there exists an $N$-player game with 2 actions for each player and $N\leq 1+C + n(3+2\ln_2(d))$ such that the projection of the set of its equilibria on the first $n$ coordinates is precisely $E$.
\end{remark}


\noindent \textbf{Proof of Proposition \ref{prop:main1}}
Remark that in our construction, at any equilibrium, the payoff of Player $Y_{i0}$ is $x_i$. Hence if $E\in[0,1]^n$ our previous construction works (considering players $Y_{10},..., Y_{n0}$ as the first $n$ players of the game). For a general $E \in [-D,D]^n$, with $D> 0$, apply the construction to 
\[E':=\{x \in \mathbb{R}^n \, | -D+2Dx\in E\}\subset[0,1]^n,\] replacing the payoff of Player $Y_{i0}$ by  the strategically equivalent payoff:
%
\[\begin{array}{c|c} 
T  &  \begin{array}{c} -D+2Dx_{i0}  \vspace{0.05cm}\\ \end{array} \\ 
\hline
B & -D+2Dx_i
\end{array} 
\]
\section{Comparison to results in \cite{LE15}}
\label{sec:comp}
We compare here the present paper to \cite{LE15}  which was written independently by Yehuda John Levy. Briefly speaking, Proposition \ref{prop:main2} is less general but the proof is constructive and gives an explicit bound on the number of additional players. More precisely, Yehuda John Levy obtain three different kind of results related to Proposition \ref{prop:main2}, that we give with increasing order of generality.

a) Our Proposition \ref{prop:main2} is ``Theorem 3.1 with binary players" in his paper. His proof is fundamentally different from ours, relying on the triangulation of semi-algebraic sets, and does not provide bounds.

b) A slight generalization is to drop the condition that the original players are binary (but still requiring the additional players to be): this is his Theorem 3.1. His proof relies on the previous binary case and an additional Lemma 4.15. He also remarks that the bounds in our Proposition \ref{prop:main2}, combined with his Lemma 4.15, yield bounds in this more general case. 

For the sake of completeness we point out that our explicit construction extends directly to that case, yielding a very slightly better bound that the one obtained in \cite{LE15}. Consider n finite sets $\mathscr{T}_i$ of cardinal $t_i$, a nonempty closed semi-algebraic set $E$ in $\Pi_{i=1}^n\Delta(\mathscr{T}_i)$, and denote by $x_{i,j}$ the probability that Player $X_i$ assigns to its $j$th pure action $\tau_j^i$. As in the proof of Proposition \ref{prop:main2}, add binary players $X_{i,j,k}$ that will play $T$ with probability $x_{i,j}^{2^k}$ at equilibrium, add players $S_{ab}$ and $U$ according to the definition of $E$, and choose an arbitrary $z$ in $E$. The payoff of player $X_i$ is now defined as $(z_{i,j}-x_{i,j,0})u$ when playing pure action $\tau_j^i$ and all the elements of the proof still hold (note that at any equilibrium in which $u=1$, $z_{i,j}<x_{i,j}=x_{i,j,0}$ would imply $x_{i,j}=0$, a contradiction). This gives a bound of $1+AB+n+2(1+\ln_2(d))\sum_i t_i$ for the number of players. Actually, using $\sum_j x_{i,j}=1$, one can replace $t_i$ by $t_i-1$ in this formula, recovering the bound $N\leq 1+AB + n(3+2\ln_2(d))$ when $t_i=2$ for all $i$. 

c) Levy's Theorems 3.2 and 3.3 are far more general, dealing with representations of semi-algebraic functions or correspondences. It does not seem that our techniques can be applied in this more general context.

We also point out that Yehuda John Levy obtains interesting results in two other directions : when the additional players are not restricted to be binary (Theorem 5.5), and when there are countably many original players (Theorems 5.3 and 5.4).
\section{Extensions}
 \label{sec:ext}
In the previous construction, if all $P_{ab}$ have coefficients in $\mathds{Q}$, the constructed game may not have all pure profile payoffs in $\mathds{Q}$, since there is no reason one can find $z\in E\cap \mathds{Q}^n$. In this section we show that it is however possible to construct such a game, at the cost of additional players. The natural idea is to choose a $z\in E$ whose coordinates $z_i$ are algebraic, with minimal polynomial $R_i\in \mathds{Q}[X]$, and to change the payoff of Player $X_i$ to  
\[\begin{array}{c|c} 
T  &  \begin{array}{c} f_i(\hat{x}) u  \vspace{0.05cm}\\ \end{array} \\ 
\hline
B & 0
\end{array} 
\]
where the $f_i$ are defined from the $R_i$ as in Lemma \ref{lm:fab} ; that is, such that $f_i(\hat{x})= R_i(x_i)$. There are however two difficulties to overcome:

- firstly, if we want to give a bound on the number of additional players, we need to know the degree of the $R_i$, which may be considerably larger than $d$. This is given by Lemma \ref{lem:algbound} below.

- secondly, it may well be that the $z_i$ have algebraic conjugates in $[0,1]$ that are not coordinates of elements of $E$. Thus $f_i(\hat{x})=0$ (the analog of $z_i = x_{i0}$) would imply $R_i(x_i)=R_i(z_i)=0$ but not $x_i=z_i$, hence would not lead to a contradiction.  We thus need to translate those variables in intervals in which the $z_i$ are the only zeroes of $R_i$. This will be done by means of additional players $V_i$.

We will need the two following technical lemmas to provide a precise bound on the number of additional players. The first one in a particular case of Algorithm 14.16 in \cite{BPR11}, while the second one is Theorem (B) in \cite{BM04}. Beware that in Lemma \ref{lem:algbound}, $d$ is the degree in each variable, while in the result we use (Algorithm 14.16 in \cite{BPR11}), it is the total degree.
\begin{lem}  \label{lem:algbound}
There exists a universal constant $K$ such that, for every nonempty semi-algebraic set $E=\cup_{a=1}^A\cap_{b=1}^B \{x \in \mathbb{R}^n,\ P_{ab}(x)\leq0\} \subset [0, 1]^n$ where each $P_{ab}\in\mathds{Z}[x_1,\cdots,x_n]$, is of degree at most $d \geq 2$ in each variable, and with coefficients bounded by $M$, there exists $z\in E$ such that all coordinates $z_i$ are zeroes of some $R_i\in\mathds{Z}[X]$ with degree less than $(nd)^{Kn}$ and coefficients less than  $(M+1)^{(nd)^{Kn}}$.
\end{lem}

\begin{lem}\label{lem:diffbound}
Let $P\in \mathds{Z}[X]$ be a polynomial with no multiple roots, with degree $d\geq 2$ and coefficients bounded by $M$. Then the minimal distance between two roots is greater than $\sqrt{3}(d+1)^{-\frac{2d+1}{2}}M^{1-d}> d^{-2d}M^{-d}$.
\end{lem}

We can now prove :

\begin{proposition} 
Let $d \geq 2$. There exists a universal constant $K'$ such that, for every nonempty semi-algebraic set $E=\cup_{a=1}^A\cap_{b=1}^B \{x \in \mathbb{R}^n,\ P_{ab}(x)\leq0\}$ in $[0,1]^n$, where each $P_{ab}\in\mathds{Z}[x_1,\cdots,x_n]$, is of degree at most $d$ in each variable, and with coefficients bounded by $M$, there exists a binary $N$-player game with integer pure payoffs bounded by $(nd(M+1))^{(nd)^{K'n}}$, and $N\leq 1+AB + K'n^2\ln_2(nd)$ such that the projection of the set of its equilibria on the first $n$ coordinates is precisely $E$.
\end{proposition}

\begin{proof}
Let $z=(z_1,\cdots,z_n)\in E$ such that every $z_i$ is algebraic. Let $R_i\in\mathds{Z}[X]$ be the minimal polynomial of $z_i$. By Lemma \ref{lem:algbound}, we may choose $z$ such that  $R_i$ is of degree less than $d_1=(nd)^{Kn}$ and with coefficients less than  $M_1=(M+1)^{(nd)^{Kn}}=(M+1)^{d_1}$. Define $M_2=(nd(M+1))^{(nd)^{2Kn}}= (nd(M+1))^{d_1^2} \geq d_1^{2d_1}M_1^{d_1}$ (the inequality holds provided $K \geq 1$, which we may assume without loss of generality).
Since $R_i$ is minimal it has no multiple roots hence by Lemma \ref{lem:diffbound} there exists an nonnegative integer $\alpha_i<M_2$ such that $z_i$ is the only zero of $R_i$ in $\left[\frac{\alpha_i}{M_2}, \frac{1+\alpha_i}{M_2}\right]$. Let $Q_i(x)= {M_2}^{d_1} R_i\left(\frac{x+\alpha_i}{M_2}\right)$. Then $Q_i$ has degree at most $d_1$, and its only zero in $[0,1]$ is $M_2 z_i-\alpha_i$. Moreover, for any $l\leq d$, the coefficient of $x^l$ is an integer and bounded by

\begin{eqnarray*}
M_2^{d_1}\sum_{j=l}^{d_1} \frac{M_1 \binom{j}{l} \alpha_i^{j-l}}{M_2^{j}}\leq M_1 M_2^{d_1} 2^{d_1} \leq (2(M+1))^{d_1} (nd(M+1))^{d_1^3} \leq  (nd(M+1))^{d_1 + d_1^3}:=M_3.
\end{eqnarray*}

\noindent
Finally, up to a change of sign, we may assume $Q_i(0)\geq 0$ and $Q_i(1)\leq 0$.


Define players $X_{ik}$, $S_{ab}$ and $U$ as in the proof of Proposition  \ref{prop:main2}. Consider now $n$ additional players $V_i$ that play Top with probability $v_i$ and with payoffs to be defined later on.
Denote by $D$ the integer part of $\ln_2(d_1)$ ; for any $i$ defines as in Lemma \ref{basicbis} some additional players $V_{ik}$ (playing a role with respect to the $v_i$ similar to the $X_{ik}$ with respect to the $x_i$) and $W_{ik}$ (playing a role similar to the $Y_{ik}$) for $k=0$ to $D$. Thus at equilibrium $v_{ik}=(v_i)^{2^k}$. For any  $Q_i$ one defines some multiaffine $f_i$ as in Lemma  \ref{lm:fab}. 
This allows us to define the payoff of player $V_i$ as
\[\begin{array}{c|c} 
T  &  \begin{array}{c} uf_i(v_{i0},\cdots,v_{iD}) \vspace{0.05cm}\\ \end{array} \\ 
\hline
B & 0
\end{array} 
\]


Finally, define the payoff of the $n$ original players $X_i$ as well as yet $n$ additional players $Y_i$ as 
\[
\mbox{Player }   X_i  \quad \begin{array}{c|c} 
T  &  \begin{array}{c} (1-u)x_{i0} + \displaystyle u\frac{v_i+\alpha_i}{M_2} \vspace{0.05cm}\\ \end{array} \\ 
\hline
B & y_i
\end{array} 
\qquad
\mbox{Player }   Y_i  \quad \begin{array}{c|c} 
T  &  \begin{array}{c} x_i \vspace{0.05cm}\\ \end{array} \\ 
\hline
B & (1-u)x_{i0} + \displaystyle  u\frac{v_i+\alpha_i}{M_2}
\end{array} 
\]

Let $x=(x_1,\cdots,x_n)$ be the $n$ first coordinates of an equilibrium and assume $x\notin E$. As in the proof of Proposition \ref{prop:main2}, one has $u=1$. Also, since we are at equilibrium Lemmas  \ref{basicbis} and \ref{lm:fab} imply that $Q_i(v_i)=f_i(v_{i0},\cdots,v_{iD})$.  Thus $Q_i(v_i)>0$ implies $v_i=1$ thus $Q_i(v_i)\leq 0$, while $Q_i(v_i)<0$ implies $v_i=0$ thus $Q_i(v_i)\geq 0$, in both cases a contradiction. Hence $Q_i(v_i)=0$ and $v_i=M_2 z_i-\alpha_i$. Since $u=1$, $(1-u)x_{i0}+u\frac{v_i+\alpha_i}{M_2}=z_i\in[0,1]$ , so applying Lemma \ref{basic} to the payoffs of players $X_i$ and $Y_i$ yields $x_i=z_i$ and $x=z\in E$, a contradiction.

Let now $x$ be in $E$ and consider the profile where $v_{ik}=w_{ik}=x_{ik}=y_{ik}=(x_i)^{2^k}$ for all $i$ and $k$, $v_i=y_i=x_i$, $u=0$, and $s_{ab}=0$ if $P_{ab}(x)\leq 0$ and 1 otherwise. Then all players are indifferent, except some $S_{ab}$ but those have no profitable deviation by construction.

The total number of players is less than $1+AB+n(5+2\ln_2(d))+2n(1+\ln_2(d_1))\leq 1+AB +11Kn^2\ln_2(nd)$.

The pure payoffs of all players except $X_i$ and $Y_i$ are integers; for players $X_i$ and $Y_i$ this is also the case provided one multiplies all their payoffs by $M_2$ (which gives a strategically equivalent game).
Clearly the only players which may have a large payoff are the $V_i$. Their payoffs are defined as multiaffine functions with $d_2=(1+\ln_2(d_1))$ variables in $[0,1]$ and coefficients smaller than $M_3$. In the payoff there is thus an addition of at most $2^{d_2}$ terms each less than $M_3$. The maximal payoff is thus less than \[2^{d_2} M_3\leq (nd(M+1))^{d_1} M_3 = (nd(M+1))^{2d_1+d_1^3}  \leq (nd(M+1))^{d_1^4}= (nd(M+1))^{(nd)^{4Kn}}.\]
\end{proof}

\section{Remarks}
\label{sec:rk}
\subsection{Optimality of the construction}
The bound we obtained in Section 2 on the number of additional players is almost optimal in $n$ and $d$. Precisely, we obtained a bound in $n\ln d$ and it is not possible to do better than  $\frac{n\ln d}{\ln(n\ln d)}$.
Indeed, recall the following bound on the number of connected components of a semi algebraic set \cite{CO02} :

\begin{lem}
Let $E$ be a semi-algebraic set defined by unions and intersections of $s$ polynomial inequalities in $r$ variables of total degree $D\geq 2$. Then its number of connected components is less than $(2D-1)D^{r+s-1}$.
\end{lem}
\begin{corollary}
The set of equilibria of a binary game with $N$ players has at most $2N^{7N}$ connected components.
\end{corollary}
\begin{proof}
The bound is trivial for $N=1$ or 2 so assume $N\geq 3$.
Let $x_i$ be the probability that Player $i$ plays his first action and denote the payoff of Player $i$ if he plays his first (resp. second) action $g^i(x^1,\cdots,x^{i-1},x^{i+1},\cdots,x^N)$ (resp.  $h^i(x^1,\cdots,x^{i-1},x^{i+1},\cdots,x^N)$) where $g^i$ and $h^i$ are multiaffine.
Since for instance  
\[\left(x_i > 0 \Rightarrow g^i(x^{-i}) \geq h^i(x^{-i}) \right) \Leftrightarrow \left(x_i \leq 0 \mbox{ or }   g^i(x^{-i}) \geq h^i(x^{-i}) \right),\]
the set of equilibria may be written as\[
\bigcap_{i\in N}\left( \{x_i\geq 1\}\cup\{g^i(x^{-i})\leq h^i(x^{-i})\} \right) \cap \left( \{x_i\leq 0\}\cup\{g^i(x^{-i})\geq h^i(x^{-i})\} \right) \cap \left(\{x_i \geq 0\}  \cap \{x_i \leq 1\} \right). 
\]
where we wrote, e.g., $\{x_i \geq 0\}$ instead of $\{x \in \mathbb{R}^n | x_i \geq 0\}$.  By the previous lemma, this set has at most $(2N-3)(N-1)^{N+6N-1}<2N^{7N}$ connected components.
\end{proof}

\begin{corollary}
For every $n\geq 2$ and $d\geq 2$, there exists a basic semi-algebraic set described by a polynomial inequality with $n$ variables and degree $2d$ which is not the set of equilibria (or equilibrium payoffs) of the first $n$ players of any binary game with less than $\frac{n\ln d}{7\ln(n\ln d)}$ players.
\end{corollary}

\begin{proof}
Let $(\alpha_1,\cdots,\alpha_d)$ be $d$ arbitrary numbers in $[0,1]$ and consider the set $E=\{x \in \mathbb{R}^n, P(x_1,\cdots,x_n)\leq 0\}$ where $P(x_1,\cdots,x_n):=\sum_{i=1}^n \Pi_{j=1}^d(x_i-\alpha_j)^2$. $E$ is clearly finite with cardinal $d^n$. Assume by contradiction that $E$ is the set of equilibria (or equilibrium payoffs) of the first $n$ players of some binary game $\Gamma$ with less than $N:=\frac{n\ln d}{7\ln(n\ln d)}$ players.
Since taking projections or applying the payoff functions can only decrease the number of connected components, the number of connected components of the equilibrium set of $\Gamma$ is at least $d^n$. However, by the previous corollary, it is at most
\[
2N^{7N}\leq (2N)^{7N} = \left(\frac{2n\ln d}{7\ln(n\ln d)}\right)^{\frac{n\ln d}{\ln(n\ln d)}} <\left(n\ln d\right)^{\frac{n\ln d}{\ln(n\ln d)}} =e^{n\ln d}= d^n
\]
a contradiction. The strict  inequality uses that  $2/7\ln (n \ln d) <1$ as soon as $d^n > \exp(e^{2/7}) \simeq 3,78$. 
\end{proof}

\subsection{Simplicity of the construction}

Our construction does not depend on Tarski-Seidenberg and real algebraic geometry at all, and is easily seen to take a polynomial time in the data. Of course, all this assumes that we are given certificates of semi-algebraicity and closedness (formula \ref{eq;semialg}), of nonemptiness (the coordinates of a point in $E$ or the description of the coordinates of an algebraic point in $E$), as well as a bound on $E$ in the case of Proposition \ref{prop:main1}. 

If not, it is possible to find such certificates but this uses real algebraic techniques (see for example \cite{BPR11}) and is typically at least exponential in the data. So, from a constructive viewpoint, we have cut the problem in two parts :

1) Preprocessing : find certificates. This is in full generality hard and time-consuming, but for specific examples it may be very easy and short, even if the set is itself very complex.

2) Use these certificates to construct a game. This is quite easy and not time-consuming.

%
%
%

\section*{Acknowledments} We thank Sylvain Sorin for helpful remarks and Yehuda John Levy for extensive discussions and for sending us a draft of his article ``Projection and Functions of Nash Equilibria". This research was supported by grant ANR-13-JS01-0004-01 (France).

\bibliographystyle{plain}
\bibliography{semialg}

\begin{thebibliography}{1}

\bibitem{BV14}
D.~Balkenborg and D.~Vermeulen.
\newblock Universality of {N}ash components.
\newblock {\em Games and Economic Behavior}, 86:67--76, 2014.

\bibitem{BPR11}
S.~Basu, R.~Pollack, and M.~Roy.
\newblock Algorithms in real algebraic geometry.
\newblock {\em AMC}, 10:12, 2011.

\bibitem{BE80}
J.~Bochnak and G.~Efroymson.
\newblock Real algebraic geometry and the 17th {H}ilbert problem.
\newblock {\em Mathematische Annalen}, 251(3):213--241, 1980.

\bibitem{BM04}
Y.~Bugeaud and M.~Mignotte.
\newblock On the distance between roots of integer polynomials.
\newblock {\em Proceedings of the Edinburgh Mathematical Society (Series 2)},
  47(03):553--556, 2004.

\bibitem{CO02}
M.~Coste.
\newblock An introduction to semialgebraic geometry.
\newblock {\em RAAG network school}, 145, 2002.

\bibitem{Da03}
R.S. Datta.
\newblock Universality of {N}ash equilibria.
\newblock {\em Mathematics of Operations Research}, 28(3):424--432, 2003.

\bibitem{Ja81}
M.J.M. Jansen.
\newblock Maximal {N}ash subsets for bimatrix games.
\newblock {\em Naval research logistics quarterly}, 28(1):147--152, 1981.

\bibitem{LSV11}
E.~Lehrer, E.~Solan, and Y.~Viossat.
\newblock Equilibrium payoffs of finite games.
\newblock {\em Journal of Mathematical Economics}, 47(1):48--53, 2011.

\bibitem{LE15}
Y.~J. Levy.
\newblock Projections and functions of {N}ash equilibria.
\newblock {\em Preprint}, 2015.

\end{thebibliography}

\end{document}